      \documentclass[smallextended,numbook,runningheads]{svjour3}     

      \usepackage{hyperref}
\usepackage{amssymb}
\usepackage{mathtools}      
\usepackage{amssymb,amsmath}
\usepackage{color}
\newtheorem{mydef}{Definition}[section]
\usepackage{caption}
\usepackage{subcaption}

\smartqed  
\begin{document}

\title{Non uniform weighted extended B-Spline finite element analysis  of non linear elliptic partial differential equations. 
}

\titlerunning{NU-WEBS on p-Laplacian}        

\author{B.V.Rathish Kumar         \and
        Ayan Chakraborty 
}

\authorrunning{Rathish Kumar and Chakraborty} 

\institute{B.V.Rathish Kumar  \at
              Department of Mathematics and Statistics \\ Indian Institute of Technology Kanpur \\
                            \email{bvrk@iitk.ac.in}           
           \and
           Ayan Chakraborty \at
              IIT Kanpur\\
              \email{ayancha@iitk.ac.in}
}

\date{Received: date / Accepted: date}

\maketitle

\begin{abstract}

 We propose a non uniform web spline based  finite element analysis  for elliptic partial differential equation  with the gradient type nonlinearity in their principal coefficients like p-laplacian equation and Quasi-Newtonian fluid flow equations. We discuss the well-posednes of the problems and also derive the apriori error estimates for the proposed finite element analysis and obtain convergence rate of $\mathcal{O}(h^{\alpha})$ for $\alpha > 0$.

\end{abstract}

\section{Introduction}

Finite element method is one of the popular numeraical techniques for solving partial differential equation's modeling real life problems from science and engineering. Currently  there is marked interest for meshless approach for solving boundary value problems as it significantly saves the cost and trouble of generating mesh, which infinitesimal  in many cases may turn out to be the computationally the most expensive job.Weighted extended B-splines is a finite element method (fem)  in a infinitesimal cost mesh framework. The present work on nonlinear elliptic problems is based on non-uniform weighted extended b-splines (NUWEBS)fem this which was originally proposed  by H$\ddot{o}$llig et al \cite{hollig2001weighted,hollig2003nonuniform,hollig2003finite}on trivial mesh framework.
 The p-laplacian equation  used into the design of shock free airfoil and non-Newtonian fluid flow model used in understanding  seepage through coarse grained porous media in some geological problems etc  have gradient type non linearity in their principal coefficients. They also occur  in the description of non linear diffusion and  filtration \cite{repin2000posteriori} , power law  materials \cite{murthy2015parallel} and Quasi Newtonian flows \cite{barrett1994quasi}.Earlier in a grid based framework mixed finite element methods were developed and analyzed in \cite{melenk1996partition,arbogast1992existence} for elliptic problems. Here we are concerned  with the finite element analysis in a gridless framework and provide the convergence analysis of weighted extended b-spline finite element analysis for p-laplacian equation and Quasi-Newtonian flow model.\\
An outline of the paper is as follows. We present some preliminary knowledge on the non-uniform weighted extended b-spline (WEB-S) space in section~(\ref{sec1})   and establish an optimal order a priori error estimates of p-Laplacian problem in section~(\ref{secP}).We discuss Quasi -Newtonian problems in~( \ref{secQ}).Furthermore,throughout our discussion $\Omega$ is a bounded, multiply connected domian , i.e, it may contain holes and we use the symbols $\preceq ,\succeq~\text{or}~\asymp$  instead $ \le, \ge ~\text{or} ~=$ whenever the constants are clear from the context and independent of the parameters. 
\section{Non uniform weighted extended B-Splines}
\label{sec1}
Let $\Omega_h$ be a webs fem approximation to $\Omega$ defined by $\Omega_h = \cup_i\Omega_{h_i}$
where $\Omega_{h_i}$ is a partitioning of $\Omega_h$ into a finite number of disjoint open regular
domains , each of maximum diameter bounded above by h . In addition, for
any two distinct domains, their closures are either disjoint, or have a common
boundaries. Associated with $\Omega_h$ is the finite-dimensional space $\mathcal{B}_h$ ( see below)\\

We can approximate a function on a bounded domain $\Omega \subset \mathbb{R}^n$ by forming a spline, i.e., a linear combination of all relevant B-splines
\begin{equation}
b_k ~~, k \in \mathbb{K} \nonumber
\end{equation}
which have some support in $\Omega$. Depending on the degree, this yields approximations of arbitrary order
and smoothness. However, numerical instabilities may arise due to the outer B-splines
\begin{equation}
b_j~~,j \in \mathbb{J} \nonumber
\end{equation}
for which no complete grid cell of their support lies in $\Omega$. Here and in the sequel, a grid cell is an interval
which in every coordinate direction is bounded by two consecutive, but different knots, and an inner grid
cell is a grid cell whose interior is completely contained in $\Omega$. A further difficulty is that, in general,
splines do not conform to homogeneous boundary conditions, which is essential for standard finite
element schemes \cite{zienkiewicz2000finite} or for matching boundaries in data fitting problems.
Fortunately, both problems can be resolved. A stable basis is obtained by forming appropriate
extensions of the inner B-splines
\begin{equation}
b_i~~, i \in \mathbb{I} = \mathbb{K}\sim \mathbb{J} \nonumber
\end{equation}
which have at least one inner grid cell in their support. Readers are refer to \cite{hollig2003finite,chaudhary2015web,de2013box} for details description.

\begin{figure}
\centering
\begin{minipage}{.5\textwidth}
  \centering
  \includegraphics[width=5cm, height=4 cm]{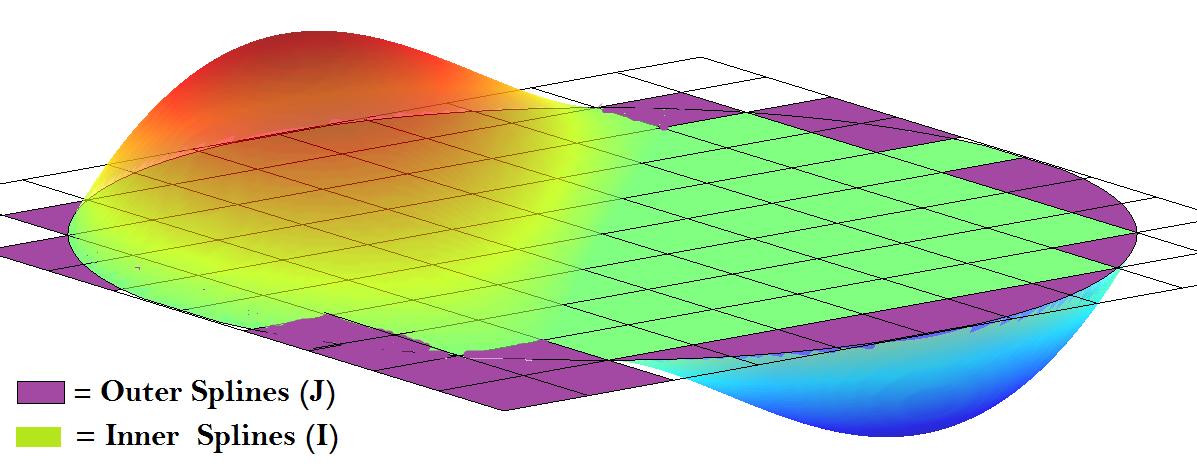}
  \captionof{figure}{WEB-Splines }
  \label{fig:test1}
\end{minipage}%
\begin{minipage}{.5\textwidth}
  \centering
  \includegraphics[width=5cm, height=4 cm]{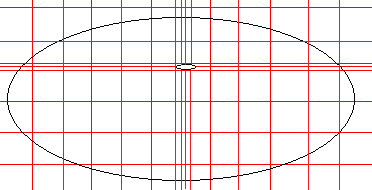}
  \captionof{figure}{Non-uniform WEBS}
  \label{fig:test2}
\end{minipage}
\end{figure}

\subsection{Splines on Bounded Domains}
The Splines $\mathbb{B}^n_h(D)$ on a bounded domain D$\subset$ $\mathbb{R}^m$ consist of all linear combinations $\sum\limits_{k\epsilon \mathbf{K}} c_k b^n_{k,h}$ of relevant B-Splines; i.e, the set $\mathbf{K}$ of relevant indices contains all k with $b^n_{k,h}(x)\ne$0 for some x $\epsilon$ D,where $b^n_{k,h}(x)= b^n (x/h-k)$ is the scaled translates. \\

\subsection{Inner and Outer Splines}
Grid cells $\mathbb{Q}$ = h ($[0,1]^m$ + $l$) are partitioned into interior, exterior and boundary cells depending on whether $\mathbb{Q}\subseteq\bar{D}$ , the interior of $\mathbb{Q}$ intersects $\partial D$ , or $\mathbb{Q}\cap D = \phi $. Among the relevant B-Splines,\textit{$b_k,k$}$\varepsilon $ \textbf{K}, distinction  made between inner B-Splines
 \begin{center}
 \textit{$b_i$,i} $\varepsilon $ \textbf{I}
  \end{center}
  which have at least one interior cell in their support, and outer B-Splines
   \begin{center}
   \textit{$b_j$,j} $\varepsilon $ \textbf{J=K$\backslash$I}
   \end{center}
   for which supp \textit{$b_j$} consists entirely of boundary and exterior cells.\\
\begin{theorem} The Spline
\begin{center}
$p=\sum_{k \varepsilon K} q(k)b_k$
\end{center}
is a polynomial of order n on D iff $q$ is a polynomial of order n on K.\\
\end{theorem}

We assume that the boundary are smooth so that smooth solution could exist. As usual, the solution is approximated by a linear combination
\begin{center}
 $\sum_i a_i B_i$
\end{center}
of basis functions $B_i$ which vanish outside a set with diameter $\asymp$ h.Moreover, the basis functions are required to vanish on the boundary so we simply multiply by a fixed \textit{weight function w} which satisfy the criteria, and in addition \textit{w} is to be smooth and $\asymp$ dist(x,D).Readers are suggested  \cite{de1988cardinal},\cite{rvachev1995r} for  more details\\

\begin{mydef}[ Extended b (eb) splines ]
For $i\in I(j), j \in J (i)~and ~Q_j$ we denote by $p_{i,j}$ the polynomial which
agrees with $b_i~ on~ Q_j$ and define the extension coefficients
\begin{equation}
e_{i,j}= \lambda_j p_{i,j} \nonumber
\end{equation}
Then, the extended B-splines (eb-splines) are
\begin{equation}
B_i = b_i + \sum\limits_{ j \in J(i)} e_{i,j} b_j ~~ i \in I(j) \nonumber
\end{equation}
where,the set of related inner indices is defined by $I(j)$ for an outer index, and $Q_j$  is an inner grid cell which is closest to supp $b_j~ for ~j \in J$ with respect to the Hausdorff metric , conversely for an inner index we define the set of related outer indices by $J(i)$.
\end{mydef}

\begin{theorem}
\label{thm2}
eb-splines and de Boor–Fix functionals $\{ \lambda_k \}$ are bi-orthogonal,i.e.
\begin{equation}
\lambda_k B_{k'}=\delta_{k,k'}~~~k,k' \in \mathbb{Z}\nonumber
\end{equation}
In addition if $Q$ is an inner grid cell in the support of $B_k$ with length bounded by $|Q| \geq \alpha |supp B_k|$ for some constant $\alpha \in (0,1]$,then
\begin{equation}
|\lambda_k p| \leq const(n,\alpha) ||p||_{\infty,Q}~~ p \in \mathcal{P}_n  \nonumber
\end{equation}
where,$|\cdot| ~  \textrm{represents measure of a set and}~ \mathcal{P}_n$  is a linear space of polynomials of degree $\leq$ n
\end{theorem}

Generalizing the univariate definitions and results of the  to $n\geq 2$ variables is
straightforward. The arguments are completely analogous. Merely the notation needs to be adapted to
the multivariate setting.We consider a tensor product grid in $\mathbb{R}^n$ with knot sequences $ t=[t^1,\ldots,t^n]$
\begin{equation}
t^{\nu}: \ldots \leq t^{\nu}_k \leq t^{\nu}_{k+1} \leq \ldots , ~~ \nu = 1,\ldots,n  \nonumber
\end{equation}

and denote by,
\begin{equation}
b_k= b^{m_1}_{k^1,t_1}(x_1) \ldots b^{m_n}_{k^n,t_n}(x_n) ~,~ k \in \mathbb{Z}^n \nonumber
\end{equation}
the corresponding tensor product B-splines of degree $ m =[m_1,\ldots,m_n]$

\begin{mydef}
\label{def1}
Let $w$ be a positive weight function which is smooth on $\Omega$ and equivalent to some power $ r\geq 0$
of the boundary distance function
\begin{equation}
w(x) \asymp dist(x,\partial \Omega) ^r \nonumber
\end{equation}
and denote by $ x_i$ the center of an inner grid cell in supp $b_i$.Then, the weighted extended B-splines (web splines) are defined by
\begin{equation}
B_i =  \frac{w}{w(x_i)}\biggl(b_i + \sum\limits_{ j \in J(i)} e_{i,j} b_j \biggr) ~~ i \in I(j) \nonumber
\end{equation}
\end{mydef}
\begin{theorem}[Jackson's Inequality]
\label{thmJ}
Let $u\in W^{1,p}(\Omega)$ . Then
\begin{equation*}
||u-\mathcal{P}_h u||_1\le \mathcal{O}(h).
\end{equation*}
\end{theorem}
\subsection{Remark}
\begin{itemize}
\item[$\bullet$]Theorem \ref{thm2} remain valid for this new class of Splines.
\item[$\bullet$] The linear span of web-splines is the web-space $\mathcal{B}$
\item[$\bullet$]The canonical projector $\mathrm{P}_h$ onto the spline space $\mathcal{B}$ is defined as
\begin{equation}
\mathrm{P}_hf = \sum\limits_{i \in \mathbb{I}} (\Lambda_i f) B_i
\label{eq23}
\end{equation}
where,the weight functional  $\Lambda_i f = w(x_i)\lambda_i (f/w) ~,~~ i \in \mathbb{I}$, with $x_i$ as in definition \ref{def1}
\item [$\bullet$] It satisfies,
\begin{align}
\label{2}
\int_{\Omega} \nabla \cdot \left(\mathrm{P}_h u - u_h \right) v_h \,dx=0
\end{align}

\end{itemize}

\section{ Elliptic partial differential equation analysis in NUWEBSFEA framework}
As NUWEBSFEA of variable coefficient of Poisson equation (VCPEA) is not available in literature, for simplicity we begin with NUWEBSFEA of VCPEA

The computational domain is denoted by $\Omega$ and and the VCPE model is given by, 

\begin{align}
\label{eq0}
-\nabla . a(x) \nabla u(x) & = f(x) ~~\text{on}~\Omega\\
u = 0 ~~\text{on} ~\partial \Omega
\end{align}
where $u$ is the scalar potential and $f$ is the source term. In case of EEG imaging this equation can often be used under the quasi static approximation of the Maxwell's equation. Moreover in EEG the source term is like the form $f(x) = \nabla \cdot d(x)$  where , $ d : \Omega \rightarrow \mathbb{R}^m ~~, ~m = 2,3$ is a vector field that describes the neural sources as idealized electric dipoles.
\\
As usual the solution is approximated by a linear combination
\begin{align*}
u_h = \sum_i a_i B_i
\end{align*}
which vanish outside a set of diameter proportional to grid width $h$.The coefficients $a_i$ 's are determined from Galerkin system
\begin{align}
\label{eq-1}
\int a(x) \nabla B_j \nabla  u_h \,dx = \sum_i \left( \int a(x) \nabla B_j \nabla  B_i \,dx  \right) a_i = \int B_j f   
\end{align}
\begin{theorem}
Let, $ u \in H^n$ be the solution of the problem (\ref{eq0}) and $u_h =\sum_ia_i B_i$ a finite element approximation obtained by solving the Galerkin system (\ref{eq-1}) If there exist a $ \kappa >0$ such that  $a(x) \ge \kappa $ then 
\begin{align*}
||u -u_h|| _1 \preceq h^{n-1} ||u||_1
\end{align*}
\end{theorem}
\begin{proof}
The proof relies on results and techniques from \cite{hollig2003nonuniform,de2013box} and the theory of weighted approximations.Moreover, the standard error estimates for splines are crucial for our arguments. We begin by noting that,
\begin{align*}
||u|| \preceq ||a (x ) u ||~~,~~ ||w u || \preceq || a (x) w u ||
\end{align*}
We refer also to \cite{de1988cardinal,zienkiewicz2000finite} where a weaker version of theorem was obtained, for some  of the preleminary arguments. Finally by using Cea's Lemma , the error of $u_h$ can be bounded , up to a constant factor, by the error of the best approximation from the finite element subspace. 
\end{proof}

\section{p-Laplace Problem}
\label{secP}
We consider the $p$- Laplacian system : Given $p\in (0, \infty)~,~f\in L^2(\Omega)~,~g\in W^{1-1/p,p}(\Omega)$ and $\inf \{ a(x), b(x)\} >0$ find $u$ such that,
\begin{equation}
\label{eq1}
-\nabla \cdot (| a(x)\nabla u|^{p-2} \nabla u) + b(x) u = f ~~~\Omega \subset \mathbb{R}^2~~u=g ~~\text{on}~ \partial \Omega
\end{equation}
For the convenience sake we assume $a(x)= b(x) = 1$. The weak formulation is given by:\\
\vspace{3mm}
Find $u \in W_g^{1,p}(\Omega) \equiv \{ v \in W^{1,p}(\Omega) ~:~ v=g~on~\partial \Omega\}$  such that,
\begin{equation}
\label{eq2}
\mathcal{L}(u,v)=\int_{\Omega} |\nabla u|^{p-2} \nabla u \cdot \nabla v \,d x + \int_{\Omega} u v \, dx= \int_{\Omega} fv \,d x ~~~\forall ~v \in W_0^{1,p}(\Omega),
\end{equation}
where, $|v|^2=\langle v,v \rangle _{\mathbb{R}^2}$.\\

Define a strictly convex functional, $J: W^{1,p}(\Omega) \longrightarrow \mathbb{R}$
\begin{align*}
J(v):= \frac{1}{p} \int_{\Omega} |\nabla v |^p \,dx +\frac{1}{2} \int_{\Omega} v^2 \,dx - \int_{\Omega} fv \,dx 
\end{align*}
Assuming, $J': W^{1,p}(\Omega) \longrightarrow \left(W^{1,p}(\Omega) \right)'$  we have,
\begin{align*}
\langle J'(u) , v \rangle = a(u,v) - \langle f , v \rangle
\end{align*} 
Define quasi-norm for $ u \in W^{1,p}(\Omega)$
\begin{align*}
|u|^2_{(u_s,p,2)}= \int_{\Omega} \left(|\nabla u_s|+ |\nabla u| \right)^{p-2} | \nabla u |^2 \,dx
\end{align*}
where $u_s$ is the solution of the problem.

\begin{theorem}
\label{thmJ1}
We have for $p \in (1,2]$
\begin{align*}
|u|^2_{(u_s,p,2)} \le |u|^p_{W^{1,p}(\mathcal{B}_h)} \le C [|u_s|_{W^{1,p}(\mathcal{B}_h)}+ |u|_{W^{1,p}(\mathcal{B}_h)}]^{2-p} |u|^{1/2}_{(u_s,p,2)}
\end{align*}
\proof see \cite{barrett1993finite}
\qed

\end{theorem}
  \subsection{Discretization using  non uniform WEB-splines basis}
 Let, $\mathfrak{T}_h$ be a quadrangulation of the $\Omega$. Let,
 \begin{align*}
 \mathcal{V}_{h}^{(n)} := \left \{ v_h \in \mathcal{C}^0(\overline{\Omega}) ~|~ v_h \mid_{\partial \Omega} =0 ~\wedge~ v_h\mid_{\tau} \in Q_n(\tau) ~\forall \tau \in \mathfrak{T}_h  \right\}
 \end{align*}
 where $Q_n(\tau)$ is the WEB-Spline space of degree $n$ defined on each cell $\tau$ . The approximation is then to seek $u_h \in \mathcal{V}_h^{(n)}$ such that 
 \begin{align}
 \label{eq7}
 \mathcal{L}(u_h,v_h)= (f, v_h) ~~\forall v_h \in \mathcal{V}_h^{(n)} 
 \end{align}
 We write ,
 \begin{align}
 \label{eq8}
 u_h =\sum_{i=1}^N c_iB_i~~c_i \in \mathbb{R}
 \end{align}
 The coefficients $c_i$  can be obtained from the following system after linearization.Consequently from  the equations (\ref{eq2}) , (\ref{eq7}) and (\ref{eq8}) we have,
 \begin{align*}
 \sum_{i=1}^N \left[ \sum_{\delta \in \mathfrak{T}_h} \int_{\delta} | \nabla u_h|^{p-2} \left(\frac{\partial B_i}{\partial x} \frac{\partial B_j}{\partial y} +\frac{\partial B_i}{\partial y} \frac{\partial B_j}{\partial x} \right) + \int_{\Omega} B_i B_j \,dx \right] = \int_{\Omega} f B_j ~~j= 1,2,\ldots, N 
 \end{align*}
\subsection{Error bounds}
\label{sec3}
The finite element approximation of \ref{eq2} that we wish to consider is: Find $u_h  \in \mathcal{B}^g_h$ and $v_h  \in \mathcal{B}^0_{h}$ such that
\begin{equation*}
\mathcal{L}(u_h,v_h)=\int_{\Omega} |\nabla u_h|^{p-2} \nabla u_h \cdot \nabla v_h \,dx + \int_{\Omega} u_h v_h \,dx= \int_{\Omega} fv_h \,dx
\end{equation*}

 where,
 \begin{align*}
 \mathcal{B}^g_h :&= \{u_h \in \mathcal{B}_h ~:~u_h=g ~\text{on} \partial \Omega_h \}\\
 \mathcal{B}^0_h :&= \{u_h \in \mathcal{B}_h ~:~u_h=0 ~\text{on} \partial \Omega_h \}
 \end{align*}
 The following error bounds
\begin{equation*}
|| \nabla (u - u_h)||_{1,p} \le \left\{
	\begin{array}{ll}
		C h^{1/(3-p)}  & \mbox{if } p \leq 2 \\
		C h^{1/(p-1)} & \mbox{if } p \ge 2
	\end{array}
\right.
\end{equation*} 
 
 \begin{equation*}
\mathcal{L}(u,w)- \mathcal{L}(v,w) \le 
 \left\{
	\begin{array}{ll}
		\nu || \nabla ( u -v) || _{0,p}^{p-1} || \nabla w||_{0,p} & \mbox{if } 1< p \leq 2 \\
		\nu \left[ C_1+C_2( || \nabla u || _{0,p}+|| \nabla v ||_{0,p})^{2-p} \right] || \nabla ( u - v)||_{0,p} || \nabla w ||_{0,p} & \mbox{if }  2\le p < \infty
	\end{array}
\right.
\end{equation*} 
$C$ depends on the domain $\Omega$ and the degree of the polynomials
  were proved in Glowinski and Marrocco\cite{glowinski1975approximation}  for the case $\Omega_h=\Omega~and~ g =0$. In this paper we are improving the error bound by employing an approach of Chow \cite{chow1989finite} and Tyukhtin \cite{tyukhtin1984rate}  in the framework of NUWEBS.\\
  We now state an important theorem which is relevant to provide a sharper error estimates.
  \begin{theorem}
  If $u$ and $u_h$ be the weak and approximate solutions then for some $C>0$ we have
  \begin{align*}
  |u -u_h|_{(u_s,p)} \le C \inf_{v_h \in \mathcal{B}_h} |u-v_h|_{(u_s,p)}
  \end{align*}
\end{theorem} 
\proof see \cite{thomas}
\qed  
  \begin{theorem} \label{thm}Let, $u$  be  the  weak solution and $u_h$ be the NU-WEBS based solution of the equation. For, $1 < p< 2$
  \begin{equation*}
|u-u_h|_{(u_s,p)} \le
\left\{
	\begin{array}{ll}
		\mathcal{O}(h^{p/2})  & \mbox{whenever } u \in W^{2,p}(\Omega)  \\
		\mathcal{O}(h) & \mbox{whenever } u \in \mathcal{C}^{2,2/p-1}(\bar{\Omega}) \cap W^{3,1}(\Omega)
	\end{array}
\right.
\end{equation*} 
Again for $2 < p < \infty $
\begin{align*}
|u-u_h|_{(u_s,p)} \le \mathcal{O}(h^{\alpha/ 2}) & \mbox { whenever } u \in W^{1,\infty}(\Omega) \cap W^{2,\alpha}(\Omega)~~,~1 \le \alpha \le 2 
\end{align*}
\end{theorem}
\begin{proof}

For the case $ 1 < p <2$. We have for $ u \in W^{2,p}(\Omega)$
\begin{align*}
|u-\mathcal{P}_hu|^2_{(u_s,p)} & = \int_{\Omega} \left( |\nabla u_s |+|\nabla ( u - \mathcal{P}_h u|\right) ^{p-2} | \nabla ( u - \mathcal{P}_h u ) |^2 \,dx\\
& \le \int_{\Omega} | \nabla ( u - \mathcal{P}_h u ) |^p \,dx ~~~~\mbox{  from Theorem} (\ref{thmJ1})\\
& \le \mathcal{O}(h^p)  ~~\mbox{from Theorem} (\ref{thmJ})\\ 
\end{align*}
Again for $ u \in C^{2,2/(p-1)}(\bar{\Omega}) \cap W^{3,1}(\Omega)$
\begin{align*}
&|\nabla ( u - \mathcal{P}_h u)| \le  C h| \mathcal{H}[u] |_{0,\infty,\mathcal{B}_h}\le Ch \mathcal{H} [u] + \mathcal{O}(h^{2/p})~~\mbox{where } \mathcal{H}[u]= |u_{xx}|+|u_{yy}|+|u_{xy}|\\
& \text{We know  for non negative  }x~, ~\psi_{\lambda}(x)= ( \lambda + x)^{p-2} x^2 ~\text{ is non decreasing.  } \\
&\text{ Therefore,}\\
& | u - \mathcal{P}_h u |^2_{(u_s,p)} = \int_{\Omega} \left( |\nabla u_s |+|\nabla ( u - \mathcal{P}_h u|\right) ^{p-2} | \nabla ( u - \mathcal{P}_h u ) |^2 \,dx\\
& \text{Clearly, right hand side is in the form of } \psi_{\lambda}(x) \text{ using the additional property } \frac{ \psi_{\lambda}(x+y)}{2} \le  \psi_{\lambda}(x) +  \psi_{\lambda}(y)\\
& \le Ch^2 \int_{\Omega} (| \nabla u_s | + Ch \mathcal{H}[u] )^{p-2} \mathcal{H}[u]^2 \,dx+ \mathcal{O}(h^{4/p})\\
& \le \mathcal{O}(h^2)
\end{align*}
Now for the case $ 2 < p < \infty$ we proceed in this way,
\begin{align*}
& | u - \mathcal{P}_h u |_{(u_s,p)} = \int_{\Omega} \left( |\nabla u_s|+|\nabla ( u - \mathcal{P}_h u|\right) ^{p-2} | \nabla ( u - \mathcal{P}_h u ) |^2 \,dx~~;~u \in W^{1,\infty} (\Omega) \cap W^{2,\alpha}(\Omega)\\
&\le \int_{\Omega} \left( |\nabla u_s |+|\nabla ( u - \mathcal{P}_h u|\right) ^{p-\alpha} | \nabla ( u - \mathcal{P}_h u ) |^{\alpha} \,dx~~~~\mbox{for } \alpha \in [1,2]\\
&\le C |\nabla ( u - \mathcal{P}_h u) |_{0,\alpha} \le \mathcal{O}(h^{\alpha})
\end{align*}
\qed
\end{proof}

\section{Quasi-Newtonian Fluids}
\label{secQ}
We now consider the following stationary non-Newtonian problem : \\
Find $(u,p)$  such that 
\begin{align}
\label{eq5}
- \sum_{j=1}^d\partial (a |D(\mathbf{u})| ^ 2) D_{ij}(\mathbf{u})/ \partial x_j + \partial p / \partial x_i & = \phi_i ~~\text{in} ~\Omega ~~~i=1,2,\ldots,d\\
\nabla \cdot u & = 0~~\text{in} ~\Omega, \\
u&=0 ~~\text{on} ~\partial \Omega 
\end{align}
where, $D(u)$ is the rate of deformation tensor with entries
\begin{align*}
D_{ij}(u) & := \frac{1}{2} \left( \frac{\partial u_i}{\partial x_j} + \frac{\partial u_j}{\partial x_i} \right)\\
\text{and}\\
|D(u)|^2 &:=  \sum_{i,j=1}^d \left[ D_{ij}(u) \right] ^ 2
\end{align*}
$\Omega$ is bounded and connected with Lipschitz boundary and $a \in C[0,\infty)$ is a positive 
satisfying
\begin{align}
\label{1}
| a_0 - a_{\infty}| \le | a(x) - a_{\infty} | \le | a_0 - a_{\infty}| \left( 1+ |x|^{1/2}\right)^{\ell}~~~   \ell \ge 0
\end{align}
where, $a_0=a(0)$ and $a_{\infty}= \sup_{x \in [0,\infty)} a(x)$.\\

Below we introduce a non linear functional ,

\begin{align*}
\mathcal{L}(v) \equiv \int_{\Omega} \left [\int_0^{|D(v)|^2 } a(s) \,ds  \right] - \sum_{i=1}^d  \langle \phi_i,v_i \rangle
\end{align*}

$\mathcal{L}$ has the following properties :
\begin{itemize}
\item{It is Gateaux differntiable}
\item{strictly convex}
\item{$\mathcal{L}'(\cdot)$ is strictly monotone}
\item{it is coercive}
\end{itemize}

 for more details readers are suggested  \cite{baranger1993numerical} and \cite{liu1994quasi}.\\

Let, $u \in (H_0^r(\Omega))^d =X $ and $ L_0^{r'}(\Omega)=Y$ where, $1/r+1/r' = 1$. ($H_0^r$ is the space of trace zero elements of $H^r$ and $L_0^{r'}$ is the mean zero $r'$ integrable functions.)
Consequently we have to find $ X $ such that,

\begin{align*}
\mathcal{A}(u,v): &= \int_{\Omega}\sum_{i,j=1}^d a \left(|D(u)|^2 \right) D_{ij}(u) D_{ij}(v)  =\langle \phi,v \rangle \equiv \sum_{i=1}^d \langle \phi_i,v_i \rangle\\
b(u,q) & = - \int_{\Omega} q ~\nabla \cdot u
\end{align*}
A finite element discretization of (\ref{eq5}) is based on the mixed weak formulation which seeks $(u,p) \in X \times Y := (H_0^r(\Omega))^d \times L_0^{r'}(\Omega)$  such that 
\begin{align}
\label{eq8}
 \int_{\Omega} \sum_{i,j=1}^d a \left(|D(u)|^2 \right) D_{ij}(u) D_{ij}(v) - \int_{\Omega} p ~  \nabla \cdot v &:=\langle \phi, v \rangle \\
\mathcal{A}(u,v) + b(v,p) &= \langle \phi ,v\rangle ~~~\forall v \in X\\
b(u,q)&=0~~~\forall q \in Y
\end{align}
We state the LBB condition for the existence-uniquenes of the solution.
\begin{theorem} \cite{amrouche1994weighted}
If for any $r \in(1,\infty)$ there exists a positive $c(q)$ such that
\begin{align*}
\inf_{ q \in L^{r'}_0(\Omega)} \sup_{w \in \left(H^{r}_0(\Omega)\right)^d} \frac{(q, \nabla \cdot w)}{||u||_{L^{r'}(\Omega)} ||w||_{H^{r}(\Omega)}} \ge c(q) >0
\end{align*} 
 then there exists a unique solution to (\ref{eq8})
\end{theorem}
\subsection{Error bounds for NUWEBS approximation}
Let, 
\begin{align*}
&\mathcal{B}_h \subset X \cap \left( W^{1,\infty}(\Omega)\right)^d ~~,~~ Y_h \subset Y \cap L^{\infty}(\Omega)\\&\text{and}~~ \mathcal{B}_h^0= \left \{ v_h \in \mathcal{B}_h ~:~ b(v_h,q_h)=0 ~\forall ~q_h \in Y_h \right \}
\end{align*}
be finite dimensional subspaces.So the corresponding approximation  problems: finding $(u_h,p_h) \in \mathcal{B}_h \times Y_h $ such that
\begin{align}
\label{eq9}
\mathcal{A}(u_h,v_h) + b(v_h,p_h) &= \langle \phi ,v_h \rangle ~~~\forall v_h \in \mathcal{B}_h\\
b(u_h,q_h)&=0~~~\forall q_h \in Y_h
\end{align}

\begin{itemize}
 \item [$\bullet$] Analogously we have the discrete version of the LBB condition.
 \begin{align*}
 \inf_{ q_h \in Y_h} \sup_{w_h \in \mathcal{B}_h} \frac{(q_h, \nabla \cdot w_h)}{||u_h||_{L^{r'}(\Omega)} ||w_h||_{H^{r}(\Omega)}} \ge c_h(q_h) >0
 \end{align*}

\item [$\bullet$] Approximation Property on $Y_h$ : \\

There is a continuous operator $ \Pi_h : L^{s}(\Omega) \longrightarrow  Y_h $ such that , 
  \begin{align*}
 || p - \Pi_h  p|| _{s} \preceq h^s ||p||_{s} ~~~p \in H^{s}(\Omega)
 \end{align*}

\end{itemize}

\begin{theorem}
Let, $ a$ satisfies (\ref{1}) , $(u,p) \in X \times Y$ be the unique solution of (\ref{eq8}) and $(u_h, p_h) \in X_h \times Y_h$ be the unique solution of (\ref{eq9}) . If the approximation property holds    on $Y_h$ then,  

\begin{align}
\label{eq13}
|| u -u_h||_X+ ||p-p_h||_2 \le \mathcal{O}(h)
\end{align}
\end{theorem}
 \proof

We follows the technique of (\cite{barret1993finite})
\begin{align*}
|| u - u_h ||_X^2 &\le \mathcal{L}(u_h) - \mathcal{L}(u) - \langle \mathcal{L}'(u), u_h-u \rangle\\
& \le \mathcal{L}(w_h) - \mathcal{L}(u) - \langle \mathcal{L}'(u) , u_h -u \rangle\\
& \le \int_0^1 \left [ \int_{\Omega} \left|a \left( \left| D(u + \tau (\mathcal{P}_hu-u))\right|^2\right) D(u+\tau(\mathcal{P}_hu-u))- a \left( \left|D(u) \right|^2 \right)D(u) \right|  \right] \cdot |D(u-\mathcal{P}_hu)|\,d\tau \\
& + \langle \nabla ( \mathcal{P}_hu -u_h) , p -\Pi_h p \rangle\\
& \le ||u - \mathcal{P}_hu||_X^2 + || \mathcal{P}_hu-u_h||_X~||p -\Pi_h p|| _2 \\
& \le  ||u - \mathcal{P}_hu||_X^2 + \left( || u -u_h||_X + || u -\mathcal{P}_hu||_X \right) || p- \Pi_h p ||_2\\
& \le  ||u - \mathcal{P}_hu||_X^2 + \frac{1}{2}\left( || u -u_h||_X ^2  + || u -\mathcal{P}_hu||_X ^2 + 2|| p- \Pi_h p ||_2^2 \right) \\
&\text{Simplifying}\\
& \le 3 || u-\mathcal{P}_hu||_X^2 + 2|| p - \Pi_h p||^2_2
\end{align*} 
Again we have ,
\begin{align*}
&\langle p-p_h, \nabla \mathcal{P}_hu \rangle =  \langle \mathcal{A} u - \mathcal{A} u_h , \mathcal{P}_hu  \rangle\\
\implies & \langle p-\Pi_h p, \nabla \mathcal{P}_hu \rangle + \langle \Pi_h p - p_h , \nabla \mathcal{P}_hu \rangle =  \langle \mathcal{A} u - \mathcal{A} u_h , \mathcal{P}_hu  \rangle\\
\implies & \langle \Pi_h p - p_h , \nabla \mathcal{P}_hu \rangle =  \langle \mathcal{A} u - \mathcal{A} u_h , \mathcal{P}_hu  \rangle - \langle p - \Pi_h p , \nabla \mathcal{P}_hu \rangle \\
\implies & || p_h - \Pi_h p ||_2 ~ || \mathcal{P}_hu||  \le c_h^{-1} \sup \langle p_h - \Pi_h p, \nabla \mathcal{P}_hu \rangle \le c_h^{-1}\left[  \left| \langle \mathcal{A} u - \mathcal{A} u_h , \mathcal{P}_hu  \rangle \right| + \sup \left| \langle p - \Pi_h p , \nabla \mathcal{P}_hu \rangle \right| \right] \\
\implies & || p - p_h ||_2 \le \left( 1+ c_h^{-1} \right) || p - \Pi_h p ||_2 + c_h^{-1} \left[\int_{\Omega} \left|a \left( \left| D(u)\right|^2\right) D(u)- a \left( \left|D(u_h) \right|^2 \right)D(u_h) \right|^2 \right]^{1/2} \\
& \le \left( 1+ c_h^{-1} \right) || p - \Pi_h p ||_2 + c_h^{-1} || u -u_h||_X \\   
\end{align*}
Simplifying by using Jackson's Inequality and Approximation property we obtain the desired results.
\qed

\section{Conclusion}
We propose non uniform web spline based mesh free finite element method for p-Laplacian problems and quasi -Newtonian problem. We provide a priori error bounds in this context.In the future, analysis of this method will be extended to Navier Stokes
problem,or, second order wave equations and miscible dispalcement problems in
porous media by using weighted isogeometric method with NURBS basis. 

\end{document}